\documentclass[12pt]{amsart}
\usepackage{amsfonts, amsmath, amssymb, amscd, amsthm, graphicx, setspace, enumitem, color}
\usepackage{hyperref}
\usepackage{pinlabel}
\usepackage{mathtools}
\usepackage{bm}

\newtheorem{theorem}{Theorem} [section]
\newtheorem{lemma}[theorem]{Lemma}
\newtheorem{proposition}[theorem]{Proposition}
\newtheorem{corollary}[theorem]{Corollary}
\newtheorem{conjecture}[theorem]{Conjecture}
\newtheorem{example}[theorem]{Example}
\newtheorem{question}[theorem]{Question}
\newtheorem{remark}[theorem]{Remark}

\newcommand{\lift}[1]{{#1}^{\wedge}}

\theoremstyle{definition}
\newtheorem{definition}[theorem]{Definition}

\newenvironment{theorem_no_number}[1][]{\begin{trivlist}
		\item[\hskip \labelsep {\bfseries Theorem \def\temp{#1}\ifx\temp\empty  #1\else  #1\fi
			.}] \itshape}  {\end{trivlist}}
\newenvironment{prop_no_number}[1][]{\begin{trivlist}
		\item[\hskip \labelsep {\bfseries Proposition \def\temp{#1}\ifx\temp\empty  #1\else  #1\fi
			.}] \itshape}  {\end{trivlist}}

\newenvironment{conj_no_number}[1][]{\begin{trivlist}
		\item[\hskip \labelsep {\bfseries Conjecture \def\temp{#1}\ifx\temp\empty  #1\else  #1\fi			.}] \itshape}  {\end{trivlist}}
\newenvironment{question_no_number}[1][]{\begin{trivlist}
		\item[\hskip \labelsep {\bfseries Question \def\temp{#1}\ifx\temp\empty  #1\else  #1\fi			.}] \itshape}  {\end{trivlist}}

\newcommand{\R}{{\mathbb R}}
\newcommand{\N}{{\mathbb N}}
\newcommand{\Z}{{\mathbb Z}}
\newcommand{\C}{{\mathbb C}}
\newcommand{\Q}{{\mathbb Q}}

\newcommand{\cO}{\mathcal O}

\newcommand{\Inn}{\mathrm{Inn}}
\newcommand{\Aut}{\mathrm{Aut}}
\newcommand{\Sp}{\mathrm{Sp}}
\newcommand{\GL}{\mathrm{GL}}
\newcommand{\Ab}{\mathrm{Ab}}

\begin{document}

\author{Alex Evetts}
\title{Conjugacy Growth in the higher Heisenberg groups}

\begin{abstract}
	We calculate asymptotic estimates for the conjugacy growth function of finitely generated class 2 nilpotent groups whose derived subgroups are infinite cyclic, including the so-called higher Heisenberg groups. We prove that these asymptotics are stable when passing to commensurable groups, by understanding their twisted conjugacy growth. We also use these estimates to prove that, in certain cases, the conjugacy growth series cannot be a holonomic function.
\bigskip

\noindent 2020 Mathematics Subject Classification: 20F65, 20E45, 20F18.

\end{abstract}

\maketitle

\section{Introduction}\label{sec:intro}

The \emph{conjugacy growth function}, \(c_{G,S}(n)\), of a group \(G\) with respect to a finite generating set \(S\) counts the number of conjugacy classes that intersect the \(n\)-ball in the Cayley graph, just as the standard growth function counts the number of elements contained in the \(n\)-ball. The study of conjugacy growth was originally motivated by counting geodesics on Riemannian manifolds, for example in work of Margulis \cite{Margulis}, and since then has received significant attention in its own right, for example \cite{SolvableConjugacy}, \cite{LinearConjugacy}, \cite{OsinHull}. 

It is conjectured by Guba and Sapir \cite{GubaSapir} that for ``ordinary" groups, exponential standard growth implies exponential conjugacy growth. This has been verified in the case of hyperbolic, soluble, and linear groups amongst others. On the other hand, Osin \cite{Osin} has constructed groups of exponential standard growth where \emph{every} non-identity element is conjugate, and so the conjecture emphatically fails in these cases. Furthermore, Osin and Hull \cite{OsinHull} have shown that conjugacy growth fails to be a commensurability invariant in general. It is worth noting however that these counter-examples are not finitely presented.

In contrast, this paper investigates the polynomial end of the growth spectrum, namely virtually nilpotent groups (as suggested by a question of Guba and Sapir \cite{GubaSapir}). We generalise work of Babenko \cite{Babenko} to derive estimates for the conjugacy growth of a class of virtually nilpotent groups whose standard growth was studied by Stoll \cite{Stoll} and which contains the so-called higher Heisenberg groups. Some of this work appears in the author's PhD thesis \cite{Thesis}.

Section \ref{sec:prelims} introduces the necessary notions and proves some basic results. Section \ref{sec:stollgroups} summarises Stoll's classification of class 2 nilpotent groups with infinite cyclic derived subgroup, in which it is shown that such groups can be constructed from finitely many copies of the first Heisenberg group (the free nilpotent group of class 2 on 2 generators). The number of copies is called the \emph{Heisenberg rank} of the group. We also derive an explicit description of the automorphisms of such groups. In section \ref{sec:conjugacygrowth} we investigate the conjugacy growth of this class of groups. We prove the following Theorem, which follows Theorem 4.1 of \cite{Babenko}, but for a more general class of groups, and a more restricted range of metrics. Our proof uses more elementary arguments.
\begin{theorem_no_number}[\ref{thm:Hconjugacy}]
	Let $G$ be a finitely generated class 2 nilpotent group with infinite cyclic derived subgroup, with Heisenberg rank $r$. Then there exists $s\in\N$ such that
	\[c_G(n)\sim\begin{cases} n^{s+2}\log n & r=1 \\ n^{s+2r} & r>1.\end{cases} \]
\end{theorem_no_number}
In section \ref{sec:virtconjgrowth} we prove the following necessary and sufficient condition for conjugacy growth to be preserved in finite extensions, which depends on understanding the \emph{twisted conjugacy growth}, \(c^{\phi}_H\), of Definition \ref{def:twistedconj}.
\begin{prop_no_number}[\ref{prop:twistedconj}]
	Let \(H\) be any finitely generated group. Then \(c^{\phi}_H(n)\preccurlyeq c_H(n)\) for every finite-order automorphism \(\phi\in\Aut(H)\) if and only if every finite extension \(G\) of \(H\) satisfies \(c_G(n)\sim c_H(n)\).
\end{prop_no_number}
This allows us to use the description of automorphisms to demonstrate that passing to a finite index supergroup does not alter the conjugacy growth function (up to the usual equivalence).
\begin{theorem_no_number}[\ref{thm:virtH}]
	Suppose that \(H\) is a class 2 nilpotent group with infinite cyclic derived subgroup. If a group \(G\) contains \(H\) as a finite index subgroup, then \(G\) and \(H\) have equivalent conjugacy growth functions.
\end{theorem_no_number}
This result naturally suggests the question of quasi-isometric invariance.
\begin{question_no_number}[\ref{que:qi}]
	In which classes of finitely generated groups is conjugacy growth a quasi-isometry invariant?
\end{question_no_number}
As noted above, there are certainly groups which do not have this property. On the other hand, it is not hard to see that conjugacy growth is a quasi-isometry invariant of virtually abelian groups (see Proposition \ref{prop:vabQI}). We conjecture that this behaviour extends to virtually nilpotent groups in general. More specifically, we make the following conjecture, in the spirit of Theorem \ref{thm:virtH}.
\begin{conj_no_number}[\ref{conj:nilp}]
	The conjugacy growth function of a finitely generated nilpotent group depends only on the structure of its lower central series.
\end{conj_no_number}

A closely related measure of conjugacy growth is the formal power series whose coefficients are the values of the conjugacy growth function. Ciobanu, Hermiller, Holt and Rees \cite{CHHR} and Antol\'in and Ciobanu \cite{FormalConjugacyGrowth} confirmed a conjecture of Rivin \cite{Rivin1}, that the conjugacy growth series of a hyperbolic group is a rational function if and only if the group is virtually cyclic (or finite). Continuing in a similar vein, the author proved the rationality of the series for all virtually abelian groups \cite{Evetts}, and Gekhtman and Yang proved transcendence for relatively hyperbolic groups and certain acylindrically hyperbolic groups \cite{ContractingElements}. All of the above results apply to any choice of finite generating set, and lead to the following conjecture.
\begin{conjecture}[Conjecture 7.2 of \cite{CEH}]\label{conj:seriesconj}
	The conjugacy growth series of any finitely presented group that is not virtually abelian is transcendental.
\end{conjecture}
This is further supported by generating set specific calculations for, amongst others, the soluble Baumslag-Solitar groups \cite{CEH}, certain wreath products \cite{Mercier} and graph products \cite{CHM}. The asymptotic estimates derived in this paper have implications for some conjugacy growth series (see Corollary \ref{cor:hrank1}), providing further evidence for this conjecture.


\section{Preliminaries}\label{sec:prelims}

We begin with the basic definitions and results. In this paper, $\N$ will denote the non-negative integers and $\N_+$ the positive integers. We occasionally use the notation \(f(n)=\cO\left(g(n)\right)\) to indicate that there exists a constant \(C\geq1\) with \(f(n)\leq Cg(n)\) for large enough \(n\), and the notation \(f(n)=o\left(g(n)\right)\) to indicate that \(\lim_{n\to\infty}\frac{f(n)}{g(n)}=0\).

\subsection{Growth and conjugacy growth}

\begin{definition}\label{def:stangrowth}
	Let \(G\) be a finitely generated group, and \(S\) a choice of finite generating set. We write \(\{S\cup S^{-1}\}^*\) to denote the language of all words over the alphabet consisting of the elements of \(S\) and their inverses (i.e. the free monoid on \(S\cup S^{-1}\)).
	\begin{enumerate}
		\item The \emph{word length} of \(g\in G\) with respect to \(S\) is
		\[|g|_S=\min\{|w| \mid w\in \{S\cup S^{-1}\}^*,~w=_G g\}.\]
		\item The (cumulative) \emph{standard growth function} of $G$ with respect to $S$ is
		\[\beta_{G,S}(n)=\#\{g\in G \mid |g|_S\leq n\}.\]
	\end{enumerate}
\end{definition}

\begin{definition}\label{def:conjgrowth}
	Let $G$ be a finitely generated group, and $S$ a choice of finite generating set. Denote by \(\mathcal{C}_G\) the set of conjugacy classes of \(G\). 
	\begin{enumerate}
		\item For \(\kappa\in\mathcal{C}_G\), define the length of \(\kappa\) with respect to \(S\) as
		\[|\kappa|_S=\min\{|g|_S\mid g\in\kappa\}.\]
		\item The (cumulative) \emph{conjugacy growth function} of \(G\) with respect to \(S\) is
		\[c_{G,S}(n)=\#\{\kappa\in\mathcal{C}_G \mid |\kappa|_S\leq n\}\]
	\end{enumerate}
\end{definition}
Geometrically, the standard growth function counts the elements in the ball of radius \(n\) in the Cayley graph of \(G\) with respect to \(S\), and the conjugacy growth function counts the number of distinct conjugacy classes intersecting the same ball. One could also study the `strict' standard and conjugacy growth functions, by considering the sphere instead of the ball, but for most purposes their qualitative behaviour is the same.

We will use the usual notion of equivalence of growth functions.
\begin{definition}
	Let $f,g\colon\N\rightarrow\N$ be two functions. We write $f\preccurlyeq g$ if there exists $\lambda>1$ such that
	\[f(n)\leq\lambda g(\lambda n)+\lambda\] for all $n\in\N$. If $f\preccurlyeq g$ and $g\preccurlyeq f$ then we write $f\sim g$, and say that the functions are \emph{equivalent}. Note that this defines an equivalence relation. We write $f\prec g$ if $f\preccurlyeq g$ but it is not the case that \(f\sim g\).
\end{definition}

It is a classical result that the equivalence class of the standard growth function of a finitely generated group does not depend on the choice of generating set. It is less well known that the same is true of the conjugacy growth function.
\begin{proposition}\label{prop:gensetequiv}
	Let \(S\) and \(T\) be two finite generating sets for a group $G$. Then
	\begin{enumerate}
		\item \(\beta_{G,S}\sim\beta_{G,T}\), and
		\item \(c_{G,S}\sim c_{G,T}\).
	\end{enumerate}
\end{proposition}
\begin{proof}
	The proof of the first part is standard, see for example \cite{Mann}. The second statement is proved similarly.
\end{proof}
Another key fact that we will use is that standard growth is a quasi-isometry invariant.
\begin{proposition}\label{prop:betaQI}
	Let \(G\) and \(H\) be quasi-isometric groups. Then \(\beta_G\sim \beta_H\).
\end{proposition}
This is not the case for conjugacy growth.
\begin{theorem_no_number}[7.2 of \cite{OsinHull}]
	There exists a finitely generated group \(G\) and a finite index subgroup \(H\leq G\) such that \(H\) has 2 conjugacy classes while \(G\) is of exponential conjugacy growth.
\end{theorem_no_number}

We can study the growth of a subset of \(G\) with respect to the same metric. This is referred to as \emph{relative} growth. Formally, we have the following definition.
\begin{definition}
	The \emph{relative growth function} of a subset \(U\) of a group \(G\) with respect to a finite generating set \(S\) is the following function
	\[\beta_U(n)=\#\{g\in U\mid |g|_S\leq n\}.\]
\end{definition}

In the following lemma, we see that the relative growth of a finite index subgroup is equivalent to that of its cosets.
\begin{lemma}\label{lem:relgrowth}
	Let \(H\) be a finite index subgroup of \(G\), and consider the coset \(tH\) for some \(t\in G\). Then the relative growth of \(tH\) is equivalent to the growth of \(G\).
\end{lemma}
\begin{proof}
	Since \(tH\) is simply a translation of \(H\), \(H\) and \(tH\) are quasi-isometric. Since \(G\) and \(H\) are also quasi-isometric, Proposition \ref{prop:betaQI} implies that \(\beta_G\sim\beta_H\sim\beta_{tH}\).
\end{proof}

We will need the following result of Breuillard and Cornulier.
\begin{lemma}[Lemma 3.1 (2) of \cite{SolvableConjugacy}]\label{lem:BC3.1(2)}
	Let $H$ be a finite index subgroup of $G$. Then $c_H\preccurlyeq c_G$.
\end{lemma}

The next Lemma shows that conjugacy growth behaves well with respect to direct products. Here, and in the rest of the paper, we use the standard notation \([g]\) to denote the conjugacy class of a group element \(g\).
\begin{lemma}\label{lem:directprod}
	Let $G$ and $H$ be groups generated by finite sets \(S\) and \(T\) respectively. Then $c_{G\times H,S\cup T}\sim c_{G,S}\cdot c_{H,T}$.
\end{lemma}
\begin{proof}
	First, note that for $(g,h)\in G\times H$ we have \[[(g,h)]=\{(g,h)^{(x,y)}\mid(x,y)\in G\times H\}=\{(g^x,h^y)\mid(x,y)\in G\times H\}=[g]\times[h].\] Now if $|[(g,h)]|_{S\cup T}\leq n$ then there exists $(x,y)\in G\times H$ with $(x,y)(g,h)(x,y)^{-1}=u_1u_2\cdots u_l$ for some $u_i\in S\cup T$, $l\leq n$. Rearranging, we find elements $s_1,\ldots,s_{l_1}\in S$ and $t_1,\ldots,t_{l_2}\in T$, with $l_1+l_2=l$, so that \[(xgx^{-1},yhy^{-1})=s_1\cdots s_{l_1}t_1\cdots t_{l_2}.\] Therefore $|xgx^{-1}|_S\leq l_1\leq n$ and $|yhy^{-1}|_T\leq l_2\leq n$, and so $|[g]|_S\leq n$ and $|[h]|_T\leq n$. Thus $c_{G\times H}(n)\leq c_G(n)\cdot c_H(n)$.
	
	Conversely, suppose $|[g]|_S\leq n$ and $|[h]|_S\leq n$. Then there are elements $\gamma\in G$ and $\delta\in H$ such that $|\gamma g\gamma^{-1}|_S\leq n$ and $|\delta h\delta^{-1}|_T\leq n$, so $|(\gamma g\gamma^{-1},\delta h\delta^{-1})|_{S\cup T}\leq 2n$. But $(\gamma g\gamma^{-1},\delta h\delta^{-1})=(\gamma,\delta)(g,h)(\gamma,\delta)^{-1}$ and so $|[(g,h)]|_{S\cup T}\leq 2n$. Thus $c_G(n)\cdot c_H(n)\leq c_{G\times H}(2n)$, giving the required result.
\end{proof}

\subsection{Nilpotent groups}
We recall the definition of a nilpotent group in order to fix some notation. For elements \(g,h\) of some group \(G\), denote their commutator \([g,h]=ghg^{-1}h^{-1}\). For a pair of subgroups \(U,V\) of \(G\), let \([U,V]=\left\langle [u,v]\mid u\in U, v\in V\right\rangle\). For any group \(G\), let \(G^{(0)}=G\) and inductively define the \(i\)-fold commutator subgroup \(G^{(i)}=[G^{(i-1)},G]\). Recall that a group \(G\) is \emph{nilpotent of class \(c\)} if and only if \(G^{(c)}=\{1\}\) and \(G^{(c-1)}\neq\{1\}\). In particular, the nilpotent groups of class 1 are precisely the abelian groups. We write \(\mathrm{Ab}(G)=G/G^{(i)}\) for the abelianisation of \(G\).
\begin{definition}
	Let \(G\) be a finitely generated group. Then each quotient \(G^{(i)}/G^{(i+1)}\) is a finitely generated abelian group. Denote the torsion-free rank of \(G^{i}/G^{i+1}\) by \(r_i\in\N\), so that \(G^{(i)}/G^{(i+1)}\cong \Z^{r_i}\times T\) for some finite abelian group \(T\).
\end{definition}

\begin{theorem}[Bass-Guivarc'h \cite{Bass}]\label{thm:BG}
	Let \(G\) be a finitely generated nilpotent group. Then the standard growth function \(\beta_G(n)\) is equivalent to the polynomial \(n^d\) where \[d=\sum_{i=0}^{c-1} (i+1)r_i.\]
\end{theorem}
Gromov \cite{Gromov} famously proved the converse; that any group of polynomial (standard) growth is virtually nilpotent.

The asymptotic behaviour of conjugacy growth in virtually \emph{abelian} groups is well-understood, as we see in the following Proposition.
\begin{proposition}\label{prop:vabQI}
	The (cumulative) standard and conjugacy growth functions of a virtually abelian group $G$ are equivalent.
\end{proposition}
\begin{proof}
	Let \(H\) be an abelian subgroup of \(G\) of finite index. Since \(H\) is nilpotent, Theorem \ref{thm:BG} gives \(d\in\N\) with $\beta_H(n)\sim n^d$, and hence \(\beta_G(n)\sim n^d\) by Proposition \ref{prop:betaQI}. Thus $c_G(n)\preccurlyeq n^d$ (since conjugacy growth is clearly bounded above by standard growth). On the other hand, Lemma \ref{lem:BC3.1(2)} gives $c_G(n)\succcurlyeq c_H(n)\sim \beta_H(n)\sim n^d$ (since a conjugacy class in an abelian group is simply an element). Therefore $c_G(n)\sim n^d\sim\beta_G(n)$.
\end{proof}
As an immediate corollary, we see that conjugacy growth is a quasi-isometry invariant within the class of virtually abelian groups.

\subsection{Generating Functions}

We will also be interested in the formal power series associated to the standard and conjugacy growth functions. We write $\Q[[z]]$ for the ring of formal power series over a variable $z$ with coefficients in $\Q$, and $\Q[z]$ for the ring of polynomials over $z$ with rational coefficients.
\begin{definition}
	Let $G$ be a group with a finite generating set $S$. Then the \emph{standard growth series} is
	\[B_{G,S}(z) = \sum_{n=0}^{\infty} \beta_{G,S}(n)z^n \in\Q[[z]].\]
	Similarly, the \emph{conjugacy growth series} is
	\[C_{G,S}(z) = \sum_{n=0}^{\infty} c_{G,S}(n)z^n \in\Q[[z]].\]
	Here, \(z\) is a complex variable.
\end{definition}
When referring to growth functions and series, we will often suppress the subscripts when the groups and/or generating sets are clear from context.

\begin{definition}
	A series \(\Gamma(z)\in\Q[[z]]\) is said to be
	\begin{enumerate}
		\item \emph{rational} if it is an element of the field of fractions of $\Q[z]$, denoted $\Q(z)$ -- in other words, there are polynomials $p,q\in\Q[z]$ such that $\Gamma=\frac{p}{q}$;
		\item \emph{algebraic} if it is algebraic over $\Q[z]$ -- in other words, it is the root of some polynomial expression with coefficient from the ring of polynomials $\Q[z]$;
		\item \emph{transcendental} if it is not algebraic;
		\item \emph{holonomic} if it is the solution to a linear finite-order differential equation with coefficients from the ring of polynomials \(\Q[z]\) -- therefore the non-holonomic series form a proper subset of the transcendental series.
	\end{enumerate}
We will refer to this classification as the \emph{algebraic complexity} of $\Gamma(z)$.
\end{definition}

Generating functions are a well-studied topic. Duchin has written a very readable introduction \cite{DuchinSurvey} to generating functions and growth in groups. For a more rigorous treatment we use \cite{Concrete} or \cite{Stanley}. There are some slightly subtle connections between the algebraic complexity of a power series and the asymptotics of the coefficients. For example, we will use the following result of Stoll.
\begin{proposition}[Proposition 3.3 of \cite{Stoll}]\label{prop:leadingcoeff}
	Let $\Gamma(z)=\sum_{n\geq0}\gamma(n)z^n\in\Q[[z]]$, and suppose that $\lim_{n\to\infty}\frac{\gamma(n)}{n^d}=a$. Then if $a$ is an irrational (resp. transcendental) number then $\Gamma(z)$ is irrational (resp. transcendental) as a series.
\end{proposition}

The next result relates to generating functions whose coefficients are in the polynomial range. 
\begin{theorem}\label{thm:nonholo}
	Let \(\gamma\colon\N\to\N\) be strictly between polynomials, that is \(n^d\prec\gamma(n)\prec n^{d+1}\) for some \(d\in\N\). Then the series \(\sum_{n\geq0}\gamma(n)z^n\) is not holonomic.
\end{theorem}
To prove this we will need another definition.
\begin{definition}
	A function $f:\N\to\C$ is called \emph{eventually quasi-polynomial} if there exist some positive integer period $N$, threshold $T\geq0$, and polynomials $f_0,f_1,\ldots,f_{N-1}$ so that for all $n\geq T$, $f(n)=f_i(n)$ whenever $n\equiv i\mod N$.
\end{definition}
The following is an immediate consequence of Proposition 4.4.1 of \cite{Stanley}.
\begin{proposition}\label{prop:EQP}
	Let \(\gamma\colon\N\to\N\) be in the polynomial range, i.e. $\gamma(n)\leq Cn^d$ for some $C>1$, $d\in\N$. Then $\sum_{n\geq0} \gamma(n)z^n$ is rational if and only if $\gamma(n)$ is eventually quasi-polynomial.
\end{proposition}

\begin{corollary}\label{cor:rationalpolynomial}
	Suppose $\gamma:\N\to\N$ is non-decreasing and in the polynomial range as above. If $\sum_{n\geq0} \gamma(n)z^n$ is rational then $\gamma\sim n^d$ for some $d\in\N$.
\end{corollary}
\begin{proof}
	Proposition \ref{prop:EQP} implies that $\gamma$ is eventually quasi-polynomial, say with polynomials $\gamma_0,\ldots,\gamma_N$ as in the definition. The degree of each $\gamma_i$ is at least the degree of $\gamma_{i-1}$ (since large enough $n$ would otherwise violate the non-decreasing assumption). But since these polynomials cycle, the degree of $\gamma_0$ must be at least the degree of $\gamma_N$. So they all have the same degree, say $d\in\N$. Thus $\gamma$ cycles between finitely many polynomials, all equivalent to $n^d$, and so $\gamma(n)\sim n^d$.
\end{proof}
We also need the following two results (see \cite{FlajoletHolonomic}).
\begin{lemma}[P\'olya-Carlson]\label{lem:PC}
	If \(\Gamma(z)\) is a power series with integer coefficients that converges on the open unit disc, then \(\Gamma(z)\) is either rational or admits the unit circle as a natural boundary.
\end{lemma}
\begin{lemma}
	Holonomic functions necessarily have only finitely many singularities.
\end{lemma}
\begin{proof}[Proof of Theorem \ref{thm:nonholo}.]
	By the Cauchy-Hadamard Theorem, the series \(\sum\gamma(n)z^n\) converges inside the unit disc, and so Lemma \ref{lem:PC} applies. Thus if the series were holonomic, and hence had only finitely many singularities, it would be rational. But by Corollary \ref{cor:rationalpolynomial}, this contradicts the hypothesis. Thus, the series cannot be holonomic. 
\end{proof}

We finish our discussion of generating functions by noting the following results from the literature.
\begin{theorem}[\cite{Benson}]\label{thm:Benson}
	Let \(G\) be a finitely-generated virtually abelian group. Then the standard growth series of \(G\) is rational, with respect to any choice of finite generating set.
\end{theorem}
\begin{theorem}[\cite{Evetts}]
	Let \(G\) be a finitely-generated virtually abelian group. Then the conjugacy growth series of \(G\) is rational, with respect to any choice of finite generating set.
\end{theorem}

\subsection{GCD sums}
To count conjugacy classes we will need various facts about greatest common divisors of tuples of integers, starting with the following lemma of Fern\'andez and Fern\'andez.

\begin{lemma}[Section 3 of \cite{FernFern}]\label{lem:FF}	
	For $n\geq1$, let $X^{(n)}_1, X^{(n)}_2,\ldots$ be a sequence of independent random variables, uniformly distributed in $\{1,2,\ldots,n\}$. Then the expected value of the greatest common divisor of the first $s$ of these random variables behaves as follows.
	\begin{equation*}
		\mathbb{E}\left(\gcd\left(X_1^{(n)},X_2^{(n)},\ldots,X_s^{(n)}\right)\right)=\begin{cases}
			\frac{1}{\zeta(2)}\log n + C + \cO\left(\frac{\log n}{\sqrt{n}}\right) & s=2 \\
			\frac{\zeta(s-1)}{\zeta(s)} + \cO\left(\frac{\log n}{n}\right) & s\geq 3
		\end{cases}
	\end{equation*}
	where $C\geq0$ is a constant.
\end{lemma}

For our purposes we will phrase this in terms of the sum of the greatest common divisors of tuples of integers whose absolute values are at most $n$.
\begin{definition}
	We define two different $n$-balls in the free abelian group $\Z^s$.
	\begin{enumerate}
		\item Let $B^{(s)}_\square(n)=\{(x_1,\ldots,x_s)\in\Z^s\mid |x_i|\leq n\text{ for each }1\leq i\leq s\}$. That is, the $n$-ball in $\Z^s$ with respect to the `cubical' generating set \[\{(\varepsilon_1,\ldots,\varepsilon_s)\mid \varepsilon_i\in\{0,1,-1\}\}.\]
		\item Let $B^{(s)}_{\ell_1}(n)=\{(x_1,\ldots,x_s)\in\Z^s\mid \sum |x_i|\leq n\}$. That is, the $n$-ball in $\Z^s$ with respect to the generating set consisting of standard basis vectors.
	\end{enumerate}
	We will omit the superscript $s$ when it is clear which dimension we are working with.
\end{definition}
Then Lemma \ref{lem:FF} can be reinterpreted as follows.
\begin{corollary}\label{cor:gcdsum}
	Let \(\bm{x}=(x_1,\ldots,x_s)\in\Z^s\). Then
	\[\sum_{\bm{x}\in B^{(2)}_\square(n)}\gcd(\bm{x})=\frac{R_2}{\zeta(2)}n^2\log n+\cO(n^2)\] where $R_2\in\Q$, and \[\sum_{\bm{x}\in B^{(s)}_\square(n)}\gcd(\bm{x})=R_s\frac{\zeta(s-1)}{\zeta(s)}n^s + \cO(n^{s-1}\log n)\] where $R_s\in\Q$ depends on the dimension $s$.
\end{corollary}
\begin{proof}
	The sum of the values of a function over some fixed finite domain is equal to the expected value of the function over the domain, multiplied by the cardinality of the domain. The standard growth function of $\Z^s$ is equivalent to $Dn^s$ (by Theorem \ref{thm:BG}), where $D\in\R$ depends on the choice of generating set, but is always rational since otherwise Proposition \ref{prop:leadingcoeff} would imply that the standard growth series was irrational, contradicting Theorem \ref{thm:Benson}.
\end{proof}

We will also need the following generalisation of Corollary \ref{cor:gcdsum}, showing that offsetting \(\bm{x}\) by a constant does not affect the asymptotics of the GCD sum.
\begin{corollary}\label{cor:gcdoffset}
	Fix an element $\bm{a}=(a_1,\ldots,a_s)\in\Z^s$. Then
	\[\sum_{\bm{x}\in B^{(2)}_\square(n)}\gcd(\bm{x}+\bm{a})=\frac{R_2}{\zeta(2)}n^2\log n+\cO(n^2)\] and \[\sum_{\bm{x}\in B^{(s)}_\square(n)}\gcd(\bm{x}+\bm{a})=R_s\frac{\zeta(s-1)}{\zeta(s)}n^s + \cO(n^{s-1}\log n)\] where $R_2, R_s\in\Q$ are the same as in Corollary \ref{cor:gcdsum}.
\end{corollary}
\begin{proof}
	Let \(a_{\max}=\max\left(|a_1|,\ldots,|a_s|\right)\). Then we have
	\begin{align*}
		B^{(s)}_\square(\min\left(n-a_{\max},0\right)) \subseteq \{\bm{x}+\bm{a}\mid \bm{x}\in B_\square^{(s)}(n)\} \subseteq B^{(s)}_\square(n+a_{\max})
	\end{align*}
	for all \(n\) and thus
	\begin{align*}
		\sum_{\bm{x}\in B^{(s)}_\square(\min\left(n-a_{\max},0\right))}\gcd(\bm{x}) \leq \sum_{\bm{x}\in B_\square^{(s)}(n)}\gcd(\bm{x}+\bm{a}) \leq \sum_{\bm{x}\in B_\square^{(s)}(n+a_{\max})}\gcd(\bm{x}).
	\end{align*}
	Applying Corollary \ref{cor:gcdsum} gives the result, since adding a constant to \(n\) does not alter the asymptotic expressions on the right hand sides.
\end{proof}

\section{A Family of Class 2 Nilpotent Groups}\label{sec:stollgroups}
In this section we discuss the nilpotent groups of class 2 whose derived subgroup is infinite cyclic. This includes the following family.
\begin{definition}
	The \emph{(higher) Heisenberg groups} are class 2 nilpotent groups, with a parameter $r\in\N_+$, given by the following presentation.
	\begin{equation*}
		H_r=\left\langle a_1,b_1,a_2,b_2,\ldots,a_r,b_r \,\middle\vert\, \begin{tabular}{l}
			$[a_i,a_j]=[a_i,b_j]=[b_i,b_j]=1~\forall i\neq j$ \\
			$[a_i,b_i]=[a_j,b_j]~\forall i\neq j$ \\
			$[[a_i,b_i],a_j]=[[a_i,b_i],b_j]=1~\forall i, j$
		\end{tabular} \right\rangle.
	\end{equation*}
\end{definition}
The commutator subgroup \(H_r^{(1)}\) is infinite cyclic, generated by the commutator $c=[a_i,b_i]$. These groups play an important role in the story of standard growth, as they provide essentially the only known examples of growth series behaviour which depends on the choice of generating set, as we will see in Stoll's result, Theorem \ref{thm:Stoll}. Furthermore, Duchin and Shapiro have shown \cite{DuchinShapiro} that the first Heisenberg group (also known as the integer or discrete Heisenberg group) \(H_1\) has rational standard growth series with respect to any choice of finite generating set.

\subsection{Stoll's classification}
In \cite{Stoll}, Stoll classifies all finitely generated class 2 nilpotent groups with infinite cyclic derived subgroup in terms of the first Heisenberg group $H_1$. We summarise this classification below.

Let $G_1$ and $G_2$ be groups with central subgroups $Z_1$ and $Z_2$ respectively. Suppose that there exists an abelian group $Z$ and homomorphisms $\varphi_1\colon Z_1\to Z$ and $\varphi_2\colon Z_2\to Z$, and consider the product homomorphism $\varphi\colon Z_1\times Z_2\rightarrow Z$ defined by
\begin{align*}
\varphi\colon(z_1, z_2)\mapsto\varphi_1(z_1)\varphi_2(z_2).
\end{align*}
Furthermore, suppose that $\varphi$ is surjective.
\begin{definition}
	The group \[G=\frac{G_1\times G_2}{\ker\varphi}\] is called the \emph{centrally amalgamated direct product of $G_1$ and $G_2$ with respect to $\varphi$}. We will write \emph{central product} for brevity.
\end{definition}

\begin{example}
	Let $G_1\cong G_2\cong\Z^2$ be given by the presentations \[G_1=\langle x_1,y_1\mid[x_1,y_1]\rangle,~G_2=\langle x_2,y_2\mid[x_2,y_2]\rangle,\] and let $Z=\langle z\mid z^2\rangle$. Define homomorphisms from the second direct factor of each $G_i$ to $Z$ as $\varphi_i\colon \langle y_i\rangle\to Z$ given by $\varphi\colon y_i\mapsto z$. Then the centrally amalgamated direct product, $G$, of $G_1$ and $G_2$ with respect to $\varphi$ is given by the presentation \[\langle x_1,x_2,z\mid [x_1,x_2], [x_1,z], [x_2,z], z^2\rangle\cong\Z^2\times\Z/2\Z.\]
\end{example}

From now on we will deal exclusively with the case where each $Z_i$ is infinite cyclic, generated by an element $z_i$, and similarly $Z$ is infinite cyclic, generated by an element $z$. Hence, each $\varphi_i$ is determined by an integer $d_i$ as follows $\varphi_i\colon z_i\mapsto z^{d_i}$. Given pairs $(G_1,z_1)$ and $(G_2,z_2)$ such that each $\langle z_i\rangle$ is an infinite cyclic central subgroup of $G_i$, we will write $(G_1,z_1)\otimes_d(G_2,z_2)$ for the central product of $G_1$ and $G_2$, amalgamated over the subgroups $\langle z_1\rangle$ and $\langle z_2\rangle$ with $\varphi_1(z_1)=z$ and $\varphi_1(z_2)=z^d$. If $d=1$, we simply write $(G_1,z_1)\otimes(G_2,z_2)$.
\begin{lemma}[Lemma 7.1 of \cite{Stoll}]\label{lem:Stollgrps}
	Let $G$ be a finitely generated $2$-step nilpotent group with $G^{(1)}\cong\Z$. Then there exists a finitely generated infinite abelian group $G_0$, and a tuple $D=(\delta_1,\ldots,\delta_{r-1})\in(\N_+)^{r-1}$, with $\delta_i|\delta_{i+1}$ for each $i$, such that
	\begin{equation*}
	G\cong(\cdots((((G_0,z)\otimes (H_1,c_1))\otimes_{\delta_1}(H_1,c_2))\otimes_{\delta_2}(H_1,c_3))\cdots)\otimes_{\delta_{r-1}}(H_1,c_{r}),
	\end{equation*}
	where each $H_1=\left\langle a_i,b_i\mid [[a_i,b_i],a_i], [[a_i,b_i],b_i]\right\rangle$ is a copy of the first Heisenberg group, $c_i$ denotes the commutator $[a_i,b_i]$, and \(z\) generates an infinite cyclic subgroup of \(G_0\).
\end{lemma}

\begin{remark}
	Note that due to the amalgamation in \(G\), \(c_1=z\) and for each \(i\), we have \(c_i=c_1^{\delta_{i-1}}\).
\end{remark}

\begin{definition}\label{def:HD}
	Let $D=(\delta_1,\ldots,\delta_{r-1})\in(\N_+)^{r-1}$, with $\delta_i|\delta_{i+1}$ for each $i$. Then we will write
	\begin{equation*}
		H_D=(\cdots(((H_1,c_1)\otimes_{\delta_1}(H_1,c_2))\otimes_{\delta_2}(H_1,c_3))\cdots)\otimes_{\delta_{r-1}}(H_1,c_r).
	\end{equation*}
\end{definition}

Note that if $I=(1,1,\ldots,1)\in(\N_+)^{r-1}$, we can express the $r$th higher Heisenberg group as $H_r=H_I$. Since the abelian subgroup $\Gamma:=G_0/\langle z\rangle$ is central, we have the following immediate corollary of Lemma \ref{lem:Stollgrps}.
\begin{corollary}\label{cor:classification}
	Let $G$ be a finitely generated $2$-step nilpotent group with $[G,G]\cong\Z$. Then there exists a finitely generated abelian group $\Gamma$ (possibly finite or trivial) and a tuple $D=(\delta_1,\ldots,\delta_{r-1})\in(\N_+)^{r-1}$, with $\delta_i|\delta_{i+1}$ for each $i$, such that
	\begin{equation*}
	G\cong\Gamma\times H_D
	\end{equation*}
\end{corollary}
\begin{definition}
	As in \cite{Stoll}, the \emph{Heisenberg rank} of \(G\cong\Gamma\times H_D\) will refer to the number, \(r\), of copies of \(H_1\) appearing in the construction. Furthermore, we will write \(s\) for the torsion-free rank of the finitely generated abelian factor \(\Gamma\).
\end{definition}
To emphasise the importance of this class of groups, we record the following, which is the main result of \cite{Stoll}.
\begin{theorem}\label{thm:Stoll}
	If \(G\) is a class 2 nilpotent group with infinite cyclic derived subgroup and Heisenberg rank at least 2, then it possesses a generating set yielding rational standard growth series, and a generating set yielding transcendental standard growth series.
\end{theorem}

We fix some notation that we will use throughout the paper.
\begin{definition}\label{def:malcev}
Suppose that \(\Gamma\times H_D\) is torsion-free, and choose a basis \(\{z_1,\ldots,z_s\}\) for \(\Gamma\). Then the elements of \(\Gamma\times H_D\) are in bijection with words of the form \(z_1^{l_1}\cdots z_s^{l_s}a_1^{i_1}b_1^{j_1}\cdots a_r^{i_r}b_r^{j_r}c^k\) where \(l_1,\ldots,l_s,i_1,j_1,\ldots,i_r,j_r,k\in\Z\). This is known as the \emph{Mal'cev normal form}, or \emph{Mal'cev coordinates}, and a variation of it exists for all finitely generated nilpotent groups (see \cite{ClemMajZym}, or the \emph{basic commutators} of \cite{Hall}).
We will frequently represent elements of \(\Gamma\times H_D\) as \(\alpha c^k\) where \(\alpha\) is an element of the form \(z_1^{l_1}\cdots z_s^{l_s}a_1^{i_1}b_1^{j_1}\cdots a_r^{i_r}b_r^{j_r}\) and \(c=c_1=[a_1,b_1]\).
\end{definition}
For the remainder of the paper, we use the generating set \(\{z_1,\ldots,z_s,a_1,\ldots,b_r\}\) for \(\Gamma\times H_D\) and lengths of elements will be taken with respect to this generating set.

\begin{definition}
	For \(g\in\Gamma\times H_D\), we write \(\bar{g}\) for the image of \(g\) in the abelianisation \(\Ab(\Gamma\times H_D)\). With this notation, we have
	\[\Ab(\Gamma\times H_D)=\langle \bar{z}_1,\ldots,\bar{z}_s,\bar{a}_1,\bar{b}_1,\ldots,\bar{a}_r,\bar{b}_r\rangle\cong\Z^{2r+s}.\]
	The preimage of \(x\in\Ab(\Gamma\times H_D)\) under the abelianisation map is a coset of the commutator subgroup \(\langle c\rangle\). We define the \emph{canonical lift} of \(x\) to be the unique element of the preimage whose \(c\)-coordinate is zero (when expressed in Mal'cev normal form), and denote it \(\lift{x}\). Therefore if \(x=\bar{z}_1^{l_1}\cdots\bar{z}_s^{l_s}\bar{a}_1^{i_1}\bar{b}_1^{j_1}\cdots\bar{a}_r^{i_r}\bar{b}_r^{j_r}\) then \(\lift{x}=z_1^{l_1}\cdots z_s^{l_s}a_1^{i_1}b_1^{j_1}\cdots a_r^{i_r}b_r^{j_r}\in\Gamma\times H_D\).
	
	We will also write elements of \(\Ab(\Gamma\times H_D)\) as vectors in \(\Z^{2r+s}\) with respect to the basis \(\{\bar{z}_1,\ldots,\bar{z}_s,\bar{a}_1,\bar{b}_1,\ldots,\bar{a}_r,\bar{b}_r\}\), so that \(z_1^{l_1}\cdots z_r^{l_r}a_1^{i_1}\cdots b_r^{j_r}=(l_1,\ldots,l_s,i_1,\ldots,j_r)\).
\end{definition}

\subsection{Automorphisms of class 2 nilpotent groups}
We will need to understand the automorphisms of our family of nilpotent groups. First, note the following easy lemma.
\begin{lemma}\label{lem:inducedhom}
	Let $G$ be a group with a characteristic subgroup $H$. Then the natural homomorphism $G\rightarrow G/H$ induces a homomorphism $\theta\colon\Aut(G)\rightarrow\Aut(G/H)$.
\end{lemma}

\begin{definition}\label{def:omegaM}
	Let \(\Omega_r\) denote the \(2r\times 2r\) block matrix consisting of blocks of the form \(\begin{pmatrix}
		0 & 1\\ -1 & 0
	\end{pmatrix}\) on the diagonal with zeroes elsewhere. Let \(\Omega_{r,s}\) be the \((2r+s)\times(2r+s)\) matrix with \(\Omega_r\) in the bottom right corner and zeroes elsewhere.
\end{definition}
For example,
		\begin{equation*}
			\Omega_{2,1}= \begin{pmatrix}
				0 & 0 & 0 & 0 & 0 \\
				0 & 0 & 1 & 0 & 0 \\
				0 & -1 & 0 & 0 & 0 \\
				0 & 0 & 0 & 0 & 1 \\
				0 & 0 & 0 & -1 & 0 \\
			\end{pmatrix}.
		\end{equation*}
Note that the skew-symmetric bilinear form \(\Omega_{r,s}\) is precisely that given by taking commutators of pairs of elements of \(\mathrm{Ab}\left(\Gamma\times H_r\right)\) with respect to the basis \(\{\bar{z}_1,\ldots,\bar{z}_s,\bar{a}_1,\bar{b}_1,\ldots,\bar{a}_r,\bar{b}_r\}\), i.e. \([\cdot,\cdot]\colon \Z^{2r+s}\times\Z^{2r+s}\to\Z\). Let \(\alpha c^{k_1}, \beta c^{k_2}\in\Gamma\times H_r\). Then \(\overline{\alpha c^{k_1}}=\bar{\alpha}\) and \(\overline{\beta c^{k_2}}=\bar{\beta}\) are elements of \(\Z^{2r+s}\) and we have
\[[\alpha c^{k_1},\beta c^{k_2}] = [\alpha,\beta] = c^{\bar{\alpha}\Omega_{r,s}\bar{\beta}^T}.\]
\begin{definition}\label{def:M}
	Let \(\mathbf{M}<\GL_{2r+s}\Z\) be the group of matrices that either preserve or reverse the bilinear form given by \(\Omega_{r,s}\) (i.e. the group of matrices \(M\) such that \(M\Omega_{r,s}M^T=\varepsilon\Omega_{r,s}\) for \(\varepsilon\in\{1,-1\}\)). Note that if \(s=0\) then \(\mathbf{M}\cong\mathrm{Sp}(2r,\Z)\rtimes\Z/2\Z\).
\end{definition}

\begin{definition}
	Suppose that \(N=\Gamma\times H_D\) is torsion-free and let \(M\in\mathbf{M}\). Define a map \(\phi_M\colon N\to N\) as follows. For each \(1\leq i\leq s\), set \(\phi_M(z_i)=\lift{\left(\bar{z}_iM\right)}\), i.e. the canonical lift of the image of \(\overline{z}_i\) under \(M\). Similarly, set \(\phi_M(a_i)=\lift{\left(\bar{a}_iM\right)}\) and \(\phi_M(b_i)=\lift{\left(\bar{b}_iM\right)}\) for each \(1\leq i\leq r\). Then extend this map to an endomorphism of \(N\) in the usual way by setting \(\phi_M(x_1\cdots x_n) = \phi_M(x_1)\cdots\phi_M(x_n)\) for any word \(x_1\cdots x_n\) in the generators.
\end{definition}
\begin{lemma}
	The map \(\phi_M\) is a well-defined automorphism of \(N\).
\end{lemma}
\begin{proof}
	First note that for any \(x\in \Z^{2r+s}\) we have \(\overline{\lift{x}}=x\).
	It is easily checked that the relators of \(N\) are mapped to the identity and therefore \(\phi_M\) defines an endomorphism of \(N\). For example, for the relator \([z_i,a_j]\), we have
	\[\left[\lift{(\bar{z}_iM)}, \lift{(\bar{a}_jM)}\right]= c^{\bar{z}_iM\Omega_{r,s}(\bar{a}_jM)^T} = c^{\varepsilon\bar{z}_i\Omega_{r,s}\bar{a}_j^T}=1.\]
	
	Furthermore, for any \(M\), \(\phi_{M^{-1}}\) is the inverse of \(\phi_M\), which is therefore an automorphism. It is sufficient to check this on the generators, for example \[\phi_{M^{-1}}\circ\phi_M(z_i) =  \lift{\left(\overline{\lift{(\bar{z}_iM)}}M^{-1}\right)} = \lift{\left(\bar{z}_i MM^{-1}\right)} = \lift{\bar{z}_i} = z_i.\]
\end{proof}

\begin{proposition}\label{prop:autoH}
	Suppose that \(N=\Gamma\times H_D\) is torsion-free. Then the natural homomorphism \(N\to N/N^{(1)}\) induces (via Lemma \ref{lem:inducedhom}) an epimorphism \(\theta\colon\Aut(N)\to\mathbf{M}\).
\end{proposition}
\begin{proof}
	We first show that the image of the induced homomorphism \(\theta\) is contained in \(\mathbf{M}\). Let \(f\in\Aut(N)\) and write \(\theta(f)=M\in\GL_{2r+s}(\Z)\). Since the commutator subgroup \(N^{(1)}\) is characteristic, \(f\) restricts to an automorphism of \(N^{(1)}=\langle c\rangle\cong\Z\), so we have \(f\colon c\mapsto c^\varepsilon\) where \(\varepsilon\in\{-1,1\}\). 
	
	Then, for any \(\alpha c^{k_1},\beta c^{k_2}\in N\), we have 
	\begin{align*}
		f\left([\alpha c^{k_1},\beta c^{k_2}]\right) &= f\left([\alpha,\beta]\right) =  [\alpha,\beta]^{\varepsilon}
	\end{align*}
	and hence \(\left[f(\alpha),f(\beta)\right] = [\alpha,\beta]^{\varepsilon}\).
	Therefore
	\begin{align*}
		c^{\bar{\alpha}M\Omega_{2r+s}M^T\bar{\beta}^T} = c^{\varepsilon\bar{\alpha}\Omega_{2r+s}\bar{\beta}^T}
	\end{align*}
	for all \(\bar{\alpha},\bar{\beta}\in \Ab(N)\) and hence \(M\Omega_{2r+s}M^T = \varepsilon\Omega_{2r+s}\), i.e. \(M\in\mathbf{M}\) as claimed.
		
	To see that \(\mathbf{M}\) is contained in the image of \(\Aut(N)\) we note that for each \(M\in\mathbf{M}\), \(\theta(\phi_M)=M\).
\end{proof}
\begin{remark}
Although we do not need it for the arguments that follow, we note that in the special case of \(N=H_r\), we can use a straightforward generalisation of an argument of Osipov \cite{Osipov} to extend Proposition \ref{prop:autoH} to give the following short exact sequence, where \(\Inn(H_r)=\Z^{2r}\).
\[1\to \Z^{2r} \to \Aut(H_r) \to \Sp_{2r}(\Z)\rtimes\Z/2\Z \to1\]
\end{remark}
In Section \ref{sec:virtconjgrowth} we will need a more explicit description of the automorphisms of $N$.
\begin{proposition}\label{prop:autoHexplicit}
	Suppose \(N=\Gamma\times H_D\) is torsion-free, and fix $f\in\Aut(N)$. Write $\theta(f)=M\in\mathbf{M}$ for the image of \(f\) as above. Then there exists \(\varepsilon\in\{1,-1\}\) and a polynomial function $\gamma\colon\R^{2r+s}\to\R$ of degree \(2\), with $\gamma(0,\ldots,0)=0$, restricting to a function \(\Z^{2r+s}\to\Z\), such that for each element $z_1^{l_1}\cdots z_s^{l_s}a_1^{i_1}b_1^{j_1}\cdots a_r^{i_r}b_r^{j_r}c^k\in N$ we have
	\begin{align*}
		f\left(z_1^{l_1}\cdots z_s^{l_s}a_1^{i_1}b_1^{j_1}\cdots a_r^{i_r}b_r^{j_r}c^k\right) &= 
		\lift{\left(\overline{z_1^{l_1}\cdots z_s^{l_s}a_1^{i_1}b_1^{j_1}\cdots 	a_r^{i_r}b_r^{j_r}}M\right)}
		c^{\varepsilon k + \gamma(l_1,\ldots,l_s,i_1,j_1,\ldots,i_r,j_r)}\\
		&= \lift{\left((l_1,\ldots,l_s,i_1,j_1,\ldots,i_r,j_r)M\right)}c^{\varepsilon k + \gamma(l_1,\ldots,l_s,i_1,j_1,\ldots,i_r,j_r)},
	\end{align*}
	with \(\varepsilon\) and \(\gamma\) depending only on \(f\).
\end{proposition}
\begin{proof}
Proposition \ref{prop:autoH} gives us a short exact sequence \(1\to K\to \Aut(N) \xrightarrow{\theta} \mathbf{M}\to1\), for some normal subgroup \(K\lhd \Aut(N)\) (which contains but may not be equal to \(\Inn(N)\)). The automorphisms \( \{\phi_M\in\Aut(N) \mid M\in\mathbf{M}\}\) defined above form a transversal for \(\Aut(N)/K\), and hence \(f=\kappa\circ \phi_M\) for some \(\kappa\in K\) and \(M\in\mathbf{M}\).

Since \(\kappa\in K\), \(\kappa\) maps to the identity in \(\mathbf{M}\). So there exist integers \(k_i\) such that
\begin{align*}
	\kappa\colon z_1 & \mapsto z_1c^{k_1} \\
	&\vdots \\
	z_s & \mapsto z_s c^{k_s} \\
	a_1 & \mapsto a_1 c^{k_{s+1}} \\
	b_1 & \mapsto b_1 c^{k_{s+2}} \\
	&\vdots \\
	b_r & \mapsto b_r c^{k_{s+2r}}.
\end{align*}
Therefore applying \(\kappa\) to some \(z_1^{l_1}\cdots z_s^{l_s}a_1^{i_1}b_1^{j_1}\cdots a_r^{i_r}b_r^{j_r}c^k\in N\) fixes the powers of the generators \(z_1,\ldots,b_r\) and adds \({p(l_1,\ldots,j_r)}\) to the power of \(c\), where \(p\) is a linear function determined by the integers \(k_1,\ldots,k_{s+2r}\).

On the other hand, \(\phi_M\in\Aut(N)\) takes each generator to a linear combination of the generators \(z_1,\ldots,b_r\), determined by the matrix \(M\). We also have \(\phi_M(c)=c^{\varepsilon}\) where \(\epsilon\in\{1,-1\}\). Rearranging into the normal form then results in an adjustment to the power of \(c\) consisting of a sum of terms of the form \(i_\lambda j_\lambda\) with coefficients determined by \(M\). Thus the image of \(z_1^{l_1}\cdots z_s^{l_s}a_1^{i_1}b_1^{j_1}\cdots a_r^{i_r}b_r^{j_r}c^k\) under \(\phi_M\) may be expressed as 
\[\lift{\left(\overline{z_1^{l_1}\cdots z_s^{l_s}a_1^{i_1}b_1^{j_1}\cdots a_r^{i_r}b_r^{j_r}}M\right)}c^{\varepsilon k+q(l_1,\ldots,j_r)} = \lift{\left((l_1,\ldots,j_r)M\right)}c^{\varepsilon k+q(l_1,\ldots,j_r)}\]
where \(q\) is a degree \(2\) polynomial.

Letting \(x=z_1^{l_1}\cdots z_s^{l_s}a_1^{i_1}b_1^{j_1}\cdots a_r^{i_r}b_r^{j_r}c^k\), we have
\begin{align*}
	f(x) &= \kappa\circ\phi_M(z_1^{l_1}\cdots z_s^{l_s}a_1^{i_1}b_1^{j_1}\cdots a_r^{i_r}b_r^{j_r}c^k)\\
	&= \kappa\left(\lift{\big((l_1,\ldots,l_s,i_1,j_1,\ldots,i_r,j_r)M\big)}c^{\varepsilon k+q(l_1,\ldots,j_r)}\right)\\
	&= \lift{\big((l_1,\ldots,l_s,i_1,j_1,\ldots,i_r,j_r)M\big)}c^{\varepsilon k+\gamma(l_1,\ldots,j_r)}
\end{align*}
where \(\gamma(l_1,\ldots,j_r)=q(l_1,\ldots,j_r)+p\left( (l_1,\ldots,j_r)M\right)\) is a degree 2 polynomial with coefficients determined by \(f\).
\end{proof}


\section{Conjugacy Growth of Higher Heisenberg Groups}\label{sec:conjugacygrowth}
This section is dedicated to proving Theorem \ref{thm:Hconjugacy}. First, we show that the parameter \(D\) does not affect the conjugacy growth of the group \(H_D\).

\begin{proposition}
	Let $D=(\delta_1,\delta_2,\ldots,\delta_{r-1})$ with each $\delta_i\mid\delta_{i+1}$. Then the conjugacy growth of $H_D$ is equivalent to that of $H_r$, that is, $c_{H_D}\sim c_{H_r}$.
\end{proposition}
\begin{proof}
	We will show that there exist groups $\Gamma_1$ and $\Gamma_2$, both isomorphic to $H_r$, such that $\Gamma_1\leq H_D\leq\Gamma_2$, with $[H_D\colon\Gamma_1]$ and $[\Gamma_2\colon H_D]$ both finite. Then Lemma \ref{lem:BC3.1(2)} will give the result.
	
	From the definition, $H_D$ is generated by elements $a_1,b_2,\ldots,a_r,b_r$ with $[a_i,b_i]=c_i=c^{\delta_{i-1}}$ for each $i$ (and $[a_1,b_1]=c$). Let $\gamma_i=\frac{\delta_{r-1}}{\delta_{i-1}}$ for $i>1$, and $\gamma_1=\delta_{r-1}$, and define the subgroup
	\[\Gamma_1=\langle a_1^{\gamma_1},b_1,a_2^{\gamma_2},b_2,\ldots,a_r^{\gamma_r},b_r\rangle\leq H_D.\]
	For any \(i\), we have $[a_i^{\gamma_i},b_i]=c_i^{\gamma_i}=c^{\delta_{i-1}\cdot\gamma_i}=c^{\delta_{r-1}}$, and so $\Gamma_1$ is isomorphic to $H_r$.
	For an element \(a_1^{i_1}b_1^{j_1}a_2^{i_2}b_2^{j_2}\cdots a_r^{i_r}b_r^{j_r}c^k\in H_D\) we have 
	\[a_1^{i_1}b_1^{j_1}a_2^{i_2}b_2^{j_2}\cdots a_r^{i_r}b_r^{j_r}c^k \in a_1^{m_1}a_2^{m_2}\cdots a_r^{m_r}c^n\Gamma_1\]
	where each \(m_\lambda=i_\lambda\bmod\gamma_\lambda\), and \(n=k\bmod\delta_{r-1}\).
	Therefore the following is a set of representatives for the cosets \(H_D/\Gamma_1\):
	\[\{a_1^{m_1}a_2^{m_2}\cdots a_r^{m_r}c^k\mid 0\leq m_\lambda<\gamma_\lambda, 0\leq k<\delta_{r-1}\},\]
	and so \([H_D\colon\Gamma_1]=\gamma_1\gamma_2\cdots\gamma_r\delta_{r-1}<\infty\) as required.
	
	Now let $\Gamma_2\cong H_r$ be generated by $\{d_1,e_1,d_2,e_2,\ldots,d_r,e_r\}$ and denote the commutator by $f=[d_i,e_i]$. Define a map $\phi\colon H_D\to\Gamma_2$ by its action on the generators:
	\begin{align*}
	 a_i&\mapsto d_i^{\delta_{i-1}} \\
	 b_i&\mapsto e_i.
	\end{align*}
	It is easily checked that the relators of $H_D$ are sent to the identity by $\phi$, and thus it is a well-defined homomorphism. Furthermore, if $\phi(a_1^{I_1}b_1^{J_1}a_2^{I_2}b_2^{J_2}\cdots a_r^{I_r}b_r^{J_r}c^K)=1$ then $d_1^{I_1}e_1^{J_1}d_2^{\delta_1I_2}e_2^{\delta_1J_2}\cdots d_r^{\delta_{r-1}I_r}e_r^{\delta_{r-1}J_r}f^K=1$ and so $I_1=J_1=\cdots=J_r=K=0$. Thus $\phi$ is a monomorphism. 
	Similarly to the previous argument, a set of representatives for the cosets \(\Gamma_2/\phi(H_D)\) is
	\[\{d_1^{l_1}\cdots d_r^{l_r}\mid 0\leq n_\lambda<\delta_{\lambda-1}\},\]
	and so \([\Gamma_2\colon\phi(H_D)]=\delta_1\delta_2\cdots\delta_{r-1}\).
\end{proof}

\begin{definition}
	If $\bm{x}=(x_1,x_2,\ldots,x_n)$ is any tuple of integers, we will write
	\[g(\bm{x}) = \gcd(x_1,x_2,\ldots,x_n)\]
	for the greatest common divisor of the entries of \(\bm{x}\).
\end{definition}

The next two results describe the structure of the conjugacy classes of \(H_r\).
\begin{lemma}\label{lem:conjstructure}
	Let \(\alpha c^k\) be an element of \(H_r\) in Mal'cev normal form, so that \(\alpha=a_1^{i_1}b_1^{j_1}a_2^{i_2}b_2^{j_2}\cdots a_r^{i_r}b_r^{j_r}\) and \(\overline{\alpha c^k}=\bar{\alpha}\in\Z^{2r}\).
	Then the conjugacy class represented by $\alpha c^k\in H_r$ is either a singleton, or a coset of a cyclic subgroup:
	\begin{align}\label{eq:conjclasses}
		[\alpha c^k]=\begin{cases}
			\{c^k\} & \text{if }\alpha\text{ is the identity}\\
			\alpha c^k\langle c^{g(\bar{\alpha})}\rangle & \text{otherwise.}
		\end{cases}
	\end{align}
\end{lemma}
\begin{proof}
	First note that $c^k$ is central, so its conjugacy class is $[c^k]=\{c^k\}$. Now consider non-identity $\alpha$. Since $[a_t,b_t]=c$ for each $1\leq t\leq r$, we have
	\begin{align*}
		a_t\alpha c^k a_t^{-1} = \alpha c^{k+i_t}\text{ and }b_t\alpha c^k b_t^{-1} = \alpha c^{k-j_t}.
	\end{align*}
	So we can express the conjugacy class of $\alpha c^k$ as follows:
	\begin{align*}
		[\alpha c^k] &= \left\{\alpha c^{k+ \sum_{t=1}^r(l_{t1}j_t - l_{t2}i_t)} \mid l_{t1},l_{t2}\in\Z\right\} \\
		&= \alpha c^k\left\langle c^{\gcd(j_1,-i_1,j_2,-i_2,\ldots,j_r,-i_r)}\right\rangle \\
		&= \alpha c^k\langle c^{g(\bar{\alpha})}\rangle.
	\end{align*}
\end{proof}
\begin{lemma}\label{lem:conjlength}
		Let $\alpha=a_1^{i_1}b_1^{j_1}\cdots a_r^{i_r}b_r^{j_r}\in H_r$ be a non-trivial element in Mal'cev normal form. Then the length of the conjugacy class $[\alpha c^k]$ lies in the range $[|\alpha|,|\alpha|+2]$, where \(|\alpha|=\sum_\lambda(|i_\lambda|+|j_\lambda|)\) is the word-length of \(\alpha\).
\end{lemma}
\begin{proof}
	We claim that any element $\alpha c^k$ has length at least $|\alpha|$, so $|\alpha|\leq|[\alpha c^k]|$. To see this, consider any word over the generators $a_1,b_1,\ldots,a_r,b_r$ that represents $\alpha c^k$. We can put this into normal form by collecting the powers of generators into the given order, at the cost of powers of $c$, using the identity $[a_t,b_t]=c$ for each $t$. Note that the exponent sum of each generator can never increase. So any word representing $\alpha c^k$ has at least $|i_1|$ instances of $a_1^{\pm1}$, and so on.
	
	From the structure of conjugacy classes given in Lemma \ref{lem:conjstructure}, each conjugacy class $[\alpha c^k]$ has a representative of the form $\alpha c^{-l}$ where $0\leq l< g(\bar{\alpha})$ and a representative of the form \(\alpha c^{m}\) where \(0<m\leq g(\bar{\alpha})\).
	
	Assume that each power \(i_t, j_t\) is non-negative. Let \(I=\{m\in\N\mid i_m\neq0\}\) and \(J=\{n\in\N\mid j_n\neq0\}\). Then, by definition, \(g(\bar{\alpha})\leq\min\{i_m,j_n\mid m\in I, n\in J\}\). 
	
	Suppose that there is some \(t\in I\cap J\), i.e. both \(i_t\) and \(j_t\) are non-zero. We have
	\begin{equation*}
		a_t^{i_t-l}b_ta_t^{l}b_t^{j_t-1}=_G a_t^{i_t}[a_t^{-l},b_t]b_t^{j_t} =_G a_t^{i_t}b_t^{j_t}c^{-l}
	\end{equation*}
	and so
	\begin{equation}\label{eq:permute}
		a_1^{i_1}b_1^{j_1}\cdots a_t^{i_t-l}b_ta_t^{l}b_t^{j_t-1} \cdots a_r^{i_r}b_r^{i_r} =_G \alpha c^{-l},
	\end{equation}
	and thus we can represent the element \(\alpha c^{-l}\) with a word of length \(|\alpha|\).
	
	Now suppose that \(I\cap J\) is empty. Since \(\alpha\neq 1\), there exists \(t\in\N\) with either \(i_t\neq0\) or \(j_t\neq 0\). If \(i_t\neq 0\) we have
	\begin{equation}\label{eq:addbs}
		a_t^{i_t-l}b_ta_t^lb_t^{-1} = a_t^{i_t}[a_t^{-l},b_t]=_G a_t^{i_t}c^{-l}
	\end{equation}
	and if \(j_t\neq0\) we have
	\begin{equation}\label{eq:addas}
		a_t^{-1}b_t^la_tb_t^{j_t-l} = [a_t^{-1},b_t^l]b_t^{j_t} =_G b_t^{j_t}c^{-l}.
	\end{equation}
	Now, similarly to equation \eqref{eq:permute}, we can represent the element \(\alpha c^{-l}\) with a word of length \(|\alpha|+2\), since in equation \eqref{eq:addbs} we have inserted an extra \(b_t\) and \(b_t^{-1}\), and in equation \eqref{eq:addas} we have inserted an extra \(a_t\) and \(a_t^{-1}\). Therefore in general we have the bound $|\alpha c^{-l}|\leq |\alpha|+2$. This in turn implies that $|\alpha|\leq|[\alpha c^k]|=|[\alpha c^{-l}]|\leq|\alpha|+2$.
	
	If some powers \(i_t, j_t\) are negative, we can find words analogous to equation \eqref{eq:permute} that represent either \(\alpha c^{-l}\) or \(\alpha c^m\), depending on the combination of signs. Thus the result holds for all values of \(\alpha\).
\end{proof}

\begin{theorem}\label{thm:Hconjugacy}
	Let $G=\Gamma\times H_D$ be a finitely generated class 2 nilpotent group with infinite cyclic derived subgroup, with Heisenberg rank $r$. Let \(s\) be the torsion free rank of \(\Gamma\). Then
	\[c_G(n)\sim\begin{cases} n^{s+2}\log n & r=1 \\ n^{s+2r} & r\geq2.\end{cases} \]
\end{theorem}
\begin{proof}
	By Lemma \ref{lem:directprod} and Theorem \ref{thm:BG} we have $c_G(n)\sim n^s\cdot c_{H_D}(n)\sim n^s\cdot c_{H_r}(n)$ and so it suffices to show that the conjugacy growth of $H_r$ is equivalent to $n^2\log n$ in the case $r=1$ and $n^{2r}$ otherwise. This follows from \cite{Babenko} but we provide a new argument here using more elementary methods.
	
	For a fixed non-identity $\alpha=a_1^{i_1}b_1^{j_1}\cdots a_r^{i_r}b_r^{j_r}\in H_r$, Lemma \ref{lem:conjstructure} implies that there are exactly $g(\bar{\alpha})$ conjugacy classes in the coset $\alpha\langle c\rangle$, and Lemma \ref{lem:conjlength} implies that they have length in the range $[|\alpha|,|\alpha|+2]$. The conjugacy classes of elements where \(\alpha\) is the identity are simply the elements of $\langle c\rangle$. Thus the conjugacy growth function satisfies the bounds
	\begin{align*}
		\beta_{\langle c\rangle}(n)+\sum_{\bar{\alpha}\in B_{\ell_1}(n-2)}g(\bar{\alpha})\leq c_{H_r}(n)\leq \beta_{\langle c\rangle}(n) + \sum_{\bar{\alpha}\in B_{\ell_1}(n)}g(\bar{\alpha}),
	\end{align*}
	where \(B_{\ell_1}(n)\) denotes the \(n\)-ball in \(\Z^{2r}\cong\Ab(H_r)\) with respect to the \(\ell_1\) norm.
	It is standard (and not hard to see) that $\beta_{\langle c\rangle}(n)\sim n^2$ for any $H_r$, noting for example that \([a_i^n,b_i^n]\) is a geodesic spelling of \(c^{n^2}\), with length \(4n\). Thus from Corollary \ref{cor:gcdsum} we have
	\begin{equation*}
		c_{H_r}(n)\sim\begin{cases}
			n^2\log n & r=1 \\
			n^{2r} & r\geq 2
		\end{cases}
	\end{equation*}
	which finishes the proof.
\end{proof}

\begin{remark}
	Comparing Theorem \ref{thm:Hconjugacy} to Corollary 4.2 of \cite{Babenko}, we see that Babenko proves a stronger result, covering a more general class of metrics and providing the leading coefficient of the growth function, but in a more restricted class of groups. In the case of \(r\geq 2\), we have \[c_{H_r}(n)=\frac{\zeta(2r-1)}{\zeta(2r)}R_r n^{2r} + o(n^{2r})\] where \(R_r\) is a rational number depending on the metric, and \(\zeta\) is the Riemann zeta function. Although it would contradict Conjecture \ref{conj:seriesconj}, if the conjugacy growth series of \(H_r\) turns out to be rational (respectively algebraic), then Proposition \ref{prop:leadingcoeff} would imply that \(\frac{\zeta(2r-1)}{\zeta(2r)}\) is a rational (respectively algebraic) number. As far as the author is aware, it is not known whether such a fraction is algebraic, although it seems unlikely.
\end{remark}


\section{Conjugacy Growth of Virtually Higher Heisenberg Groups}\label{sec:virtconjgrowth}
In this section we show that if \(G\) is commensurable to a higher Heisenberg group then they have equivalent conjugacy growth functions. We will need the following lemma.
\begin{lemma}\label{lem:subgrouprank}
	If \(U\) is a subgroup of \(H_1\) then the Heisenberg rank of \(U\) is at most \(1\).
\end{lemma}
\begin{proof}
	First, we claim that a pair of elements of \(H_1\) commute if and only if their images in \(\Ab(H_1)\) are colinear as vectors in \(\Z^2\). To see this, let \(g=a^ib^jc^k\) and \(h=a^Ib^Jc^K\) be a elements of \(H_1\). We have \(gh=a^{i+I}b^{j+J}c^{k+K-jI}\) and \(hg=a^{I+i}b^{J+j}c^{K+k-Ji}\) and so they commute if and only if \(jI=Ji\). This is equivalent to the vectors \(\bar{g}=(i,j)\) and \(\bar{h}=(I,J)\) being colinear (including the possibility that one or both is the zero vector).
	
	To prove the Lemma, suppose, on the contrary, that \(U\) has Heisenberg rank at least \(2\). This implies that there exist four elements \(x_1,y_1,x_2,y_2\in U\leq H_1\) such that \([x_1,y_1]\neq 1\), \([x_2,y_2]\neq1\), and 
	\begin{equation}\label{eq:xycommute}
		[x_1,x_2]=[x_1,y_2]=[y_1,x_2]=[y_1,y_2]=1.
	\end{equation}
	Now the claim above along with \eqref{eq:xycommute} implies that \(\bar{x}_1\) and \(\bar{x}_2\) are colinear, and \(\bar{y}_1\) and \(\bar{x}_2\) are colinear. Since being colinear is a transitive relation, \(\bar{x}_1\) and \(\bar{y_1}\) are colinear, and hence \(x_1\) and \(y_1\) commute, which is a contradiction.	
\end{proof}
Next, we show that conjugacy growth is preserved when passing to finite index subgroups.
\begin{lemma}\label{lem:fisubgroup}
	Let $G$ be a class 2 nilpotent group with infinite cyclic derived subgroup. If $H$ is a finite index subgroup of $G$ then $H$ is also class 2 nilpotent with infinite cyclic derived subgroup. Furthermore, $G$ and $H$ have equivalent conjugacy growth functions.
\end{lemma}
\begin{proof}
	Since $H$ is a subgroup of $G$ it is nilpotent of class at most 2 and since it has finite index, $[H,H]$ has finite index in $[G,G]$ and is therefore infinite cyclic. We have $G=\Gamma_1\times H_{D_1}$ and $H=\Gamma_2\times H_{D_2}$, where $\Gamma_1$ and $\Gamma_2$ are abelian and the groups $H_{D_1}$ and \(H_{D_2}\) are as in Definition \ref{def:HD}. Let $r_1$ and $r_2$ denote the corresponding Heisenberg ranks, and let $s_i$ denote the torsion-free rank of $\Gamma_i$.
	
	By Proposition \ref{prop:betaQI}, $G$ and $H$ have equivalent standard growth functions. Therefore by Theorem \ref{thm:BG} we have $s_1+2r_1+2=s_2+2r_2+2$, i.e. 
	\begin{equation}\label{eq:sr}
		s_1+2r_1 = s_2+2r_2.
	\end{equation}
	This also follows from Theorem \ref{thm:Pansu}. If $r_1,r_2\geq2$ then Theorem \ref{thm:Hconjugacy} implies that \(c_G(n)\sim n^{s_1+2r_1}\) and \(c_H(n)\sim n^{s_2+2r}\) and hence \(c_G\sim c_H\) by \eqref{eq:sr}. If \(r_1=r_2=1\) then \(s_1=s_2\) and Theorem \ref{thm:Hconjugacy} again implies that \(c_G\sim c_H\).

	Suppose $r_1>1$ and $r_2=1$. Applying Theorem \ref{thm:Hconjugacy}, we have \(c_G(n)\sim n^{s_1+2r_1}\) and \(c_H(n)\sim n^{s_2+2}\log n\). From \eqref{eq:sr}, \(n^{s_1+2r_1}=n^{s_2+2}\), and then \(c_G(n)\sim n^{s_2+2}\) and \(c_H(n)\sim n^{s_2+2}\log n\) together violate Lemma \ref{lem:BC3.1(2)}.
		
	Finally, suppose that \(r_1=1\) and \(r_2>1\). So there exist four elements \(x_1,y_1,x_2,y_2\in H<G=\Gamma_1\times H_1\) such that \(\langle x_1,y_1,x_2,y_2\rangle\leq G\) has Heisenberg rank \(2\). Since \(G\) is a direct product and \(\Gamma_1\) is abelian, two elements of \(G\) commute if and only if their \(H_1\) components commute. So by passing to just the \(H_1\) components of each element, we find elements \(x'_1,y'_1,x'_2,y'_2\) contained in the \(H_1\) factor of \(G\), that generate a subgroup of Heisenberg rank \(2\). This contradicts Lemma \ref{lem:subgrouprank}.
\end{proof}

Now we show that conjugacy growth is also preserved when passing to finite index supergroups.
\begin{theorem}\label{thm:virtH}
	Suppose \(H\) is a class 2 nilpotent group with infinite cyclic derived subgroup. If a group \(G\) contains \(H\) as a finite index subgroup, then \(G\) and \(H\) have equivalent conjugacy growth functions.
\end{theorem}
Before we prove this Theorem, we note the following consequence.
\begin{corollary}\label{cor:hrank1}
	If \(G\) is (virtually) nilpotent with Heisenberg rank equal to 1, then its conjugacy growth series is non-holonomic, with respect to any finite generating set.
\end{corollary}
\begin{proof}
	Theorems \ref{thm:virtH} and \ref{thm:Hconjugacy} show that the conjugacy growth of such a group is equivalent to \(n^d\log n\) for some positive integer \(d\). But by Theorem \ref{thm:nonholo}, no such function can have a holonomic power series.
\end{proof}

To prove Theorem \ref{thm:virtH}, we will use a general necessary and sufficient condition for conjugacy growth to be preserved under finite extensions. First, we define twisted conjugacy growth.
\begin{definition}\label{def:twistedconj}
	Let \(\phi\) be a fixed automorphism of a group \(G\). Define the \emph{\(\phi\)-twisted conjugacy class} of \(g\in G\) by \[[g]^{\phi} = \{\phi(h)gh^{-1}\mid h\in G\}\] and write \(\mathcal{C}^{\phi}_G\) for the set of \(\phi\)-twisted conjugacy classes of \(G\) (it is easy to check that \(\phi\)-twisted conjugacy defines an equivalence relation). As in Definitions \ref{def:stangrowth} and \ref{def:conjgrowth}, the length of a twisted conjugacy class is defined as
	\[|[g]^{\phi}|_S=\min\{|h|_S\mid h\in[g]_{\phi}\}\] and the corresponding \emph{\(\phi\)-twisted conjugacy growth function} is
	\[c^{\phi}_{G,S}(n) = \#\{\kappa\in\mathcal{C}^{\phi}_G\mid |\kappa|_S\leq n\}.\]
\end{definition}
It is not hard to prove that the analogue of Proposition \ref{prop:gensetequiv} holds in this case. In other words, the equivalence class of \(\phi\)-twisted conjugacy growth does not depend on the choice of generating set and we can therefore drop the \(S\) from the notation. The following result is reminiscent of Theorem 3.1 of \cite{BogMartVen} in that conjugacy in an extension is understood via twisted conjugacy of the base group.
\begin{proposition}\label{prop:twistedconj}
	Let \(H\) be any finitely generated group. Then \(c^{\phi}_H(n)\preccurlyeq c_H(n)\) for every finite-order automorphism \(\phi\in\Aut(H)\) if and only if every finite extension \(G\) of \(H\) satisfies \(c_G(n)\sim c_H(n)\).
\end{proposition}
\begin{proof}
	First, suppose that \(c^{\phi}_H(n)\preccurlyeq c_H(n)\) for every finite order \(\phi\in\Aut(H)\). For any finite extension \(G\), Lemma \ref{lem:BC3.1(2)} states that \(c_H(n)\preccurlyeq c_G(n)\), so it remains to show the opposite bound. Fix a choice of transversal, \(T\subset G\), for \(G/H\). We consider the cosets $tH$ separately, for each $t\in T$ (since the conjugacy class of any element of \(tH\) is contained within \(tH\) by normality). If a pair of elements of \(tH\) are conjugate by an element of \(H\) then they are in the same conjugacy class (in \(G\)). Therefore the number of \(H\)-conjugacy classes in the \(n\)-ball is at least the number of \(G\)-conjugacy classes. Therefore only considering conjugation by elements of \(H\) will give an upper bound for the \(G\)-conjugacy growth of \(tH\). Fixing $t\in T$, and choosing an element $h\in H$, we have the $H$-conjugacy class
	\[[th]_H = \{x th x^{-1}\mid x\in H\} = \{t\phi_t(x)hx^{-1}\mid x\in H\}=t[h]^{\phi_t},\]
	where $\phi_t\in\Aut (H)$ is the finite-order automorphism defined by conjugation: $\phi_t\colon\gamma\mapsto t^{-1}\gamma t$. So we can understand \(H\)-conjugation in \(tH\) as \(\phi_t\)-twisted conjugation in \(H\) itself.
	
	Since \(c^{\phi_t}_H(n)\preccurlyeq c_H(n)\) by hypothesis, and the growth of \(tH\) is equivalent to the growth of \(H\) (as per Lemma \ref{lem:relgrowth}), the contribution to the conjugacy growth of \(G\) from the coset \(tH\) is at most the conjugacy growth of \(H\). Since there are only finitely many such cosets, we have \(c_G(n)\preccurlyeq c_H(n)\) as claimed.
	
	For the converse part of the statement, we prove its contrapositive, namely that if there exists some finite order automorphism giving twisted conjugacy growth strictly greater than untwisted conjugacy growth then there is a finite extension of \(H\) with inequivalent conjugacy growth. Suppose \(\phi\in\Aut(H)\) is such an automorphism. So there exists \(k\in\N\) such that \(\phi^k\) is the identity automorphism, and we have \(c^{\phi}_H(n)\succ c_H(n)\). Form the semidirect product \(G:=H\rtimes_{\phi}\Z/k\Z\), where the finite cyclic group is generated by \(t\), which acts on \(H\) via \(t^{-1}ht=\phi(h)\) for all \(h\in H\). We have \(c_H\preccurlyeq c_G\) from Lemma \ref{lem:BC3.1(2)}, and we need to show this is a strict inequality to prove our result. It is enough to do this for one of the \(k\) cosets of \(H\) in \(G\) (which are again closed under conjugation since \(H\lhd G\)). We consider conjugating an element of the coset \(tH\) by elements of \(H\). Let \(h\in H\). As above we have
	\begin{align*}
		[th]_H = \{x thx^{-1}\mid x\in H\} = \{t\phi(x)hx^{-1}\mid x\in H\} = t[h]^{\phi}
	\end{align*}
	and so, again, \(H\)-conjugacy classes of the coset \(tH\) are \(t\)-translates of \(\phi\)-twisted conjugacy classes of \(H\). Therefore the number of \(H\)-conjugacy classes in \(tH\) which intersect the \(n\)-ball is equivalent to the twisted conjugacy growth function \(c^{\phi}_H(n)\). We have 
	\begin{align*}
		[th]_G = \{g th g^{-1}\mid g\in G\} = \{t^ixthx^{-1}t^{-i}\mid x\in H, 0\leq i\leq k-1\} = \bigcup_{i=0}^{k-1}t^i[th]_Ht^{-i}
	\end{align*}
	and so passing from \(H\)-conjugacy classes to full conjugacy can only increase the number of conjugacy classes in the \(n\)-ball by at most a constant factor of \(k\), which doesn't change the equivalence class. Thus the contribution to conjugacy growth from \(tH\) is equivalent to \(c^\phi_H(n)\), which is strictly greater than \(c_H(n)\) by hypothesis, and we have \(c_G(n)\succ c_H(n)\) as required.
\end{proof}

Now we apply this general criterion to the specific case at hand. We will need the following simple observation.
\begin{lemma}\label{lem:Ucosets}
	Let \(U\) be a subgroup of \(\Z^d\), and therefore a free abelian group of rank \(u\leq d\). Then the \(n\)-ball in \(\Z^d\) (with respect to the standard generating set of unit vectors) intersects \(\cO(n^{d-u})\) distinct cosets of \(U\).
\end{lemma}
\begin{proof}[Proof of Theorem \ref{thm:virtH}.]
	To apply Proposition \ref{prop:twistedconj}, we need a normal subgroup. Any torsion in \(H\) is contained in the abelian factor, and is therefore a finite subgroup, so we initially pass to the torsion-free direct factor. The conjugacy growth of a finite group is clearly eventually constant, so Lemma \ref{lem:directprod} shows that removing a finite direct factor preserves (the equivalence class of) the conjugacy growth function. Then, a standard argument allows us to pass to a \emph{normal} subgroup of \(G\), contained in \(H\), also of finite index. By Lemma \ref{lem:fisubgroup}, this subgroup has equivalent conjugacy growth to the original subgroup. Therefore, we may assume without loss of generality that \(H\) is normal and torsion-free. Now to prove the Theorem it is enough to show that \(c^{\phi}_H(n)\preccurlyeq c_H(n)\) for any \(\phi\in\Aut(H)\).
	
	We explicitly describe $\phi$-twisted conjugacy in \(H\), for any fixed \(\phi\in\Aut(H)\). With \(h=z_1^{x_1}\cdots z_s^{x_s}a_1^{u_1}b_1^{v_1}\cdots a_r^{u_r}b_r^{v_r}c^w\in H\) and \(x=z_1^{l_1}\cdots z_s^{l_s}a_1^{i_1}b_1^{j_1}\cdots a_r^{i_r}b_r^{j_r}c^k\in H\), Proposition \ref{prop:autoHexplicit} gives
\begin{align}
	\nonumber\phi(x)hx^{-1} 	&=\phi(z_1^{l_1}\cdots b_r^{j_r}c^k)\cdot z_1^{x_1}\cdots b_r^{v_r}c^w \cdot z_1^{-l_1}\cdots b_r^{-j_r}c^{-k-\sum_{\lambda} i_{\lambda}j_{\lambda}}\\
	\nonumber&=\lift{\Big( (l_1,\ldots,j_r)(M-I) + (x_1,\ldots, v_r) \Big)} c^{\varepsilon k -k +w -\sum_{\lambda} i_{\lambda}j_{\lambda} +\sum_{\lambda}(v_{\lambda}i_{\lambda}-u_{\lambda}j_{\lambda}) +\gamma(l_1,\ldots,j_r)}\\
	&=\lift{\Big( (l_1,\ldots,j_r)(M-I) + (x_1,\ldots,v_r) \Big)} c^{\varepsilon k -k +w +f(l_1,\ldots,j_r)} \label{eq:twistedconj}
\end{align}
where \(\gamma\) is a polynomial of degree 2 (and therefore so is \(f\)), \(\varepsilon\in\{-1,1\}\), and \(M\in\mathbf{M}\), and these all depend only on \(\phi\).
We analyse various cases depending on the value of \(\varepsilon\) and the nature of the matrix \(M-I\) (which defines an endomorphism of \(\mathrm{Ab}(H)\cong\Z^{2r+s}\)).

	\begin{enumerate}
		\item If \(\phi\) is the identity automorphism then \(\phi\)-twisted conjugacy is nothing more than standard conjugacy, and so \(c^{\phi}_H(n)=c_H(n)\) in this case. In the remaining cases we assume that \(\phi\) is not the identity.
		
		\item Suppose $\varepsilon=-1$, i.e. $\phi$ inverts the commutator. If $x=c^k$ then $\phi_t(x)hx^{-1}=hc^{-2k}$. So just considering conjugators of the form \(c^k\) for varying \(k\), we can see that each fixed $(x_1,\ldots,v_r)\in\Ab(H)$ corresponds to at most two \(\phi\)-twisted conjugacy classes in \(H\), which both have representatives of length \(|(x_1,\ldots,v_r)|\), as in Lemma \ref{lem:conjlength}. Considering all conjugators will not increase the number of conjugacy classes of a given length, so we have \(c_H^{\phi}(n)\preccurlyeq \beta_{\mathrm{Ab}(H)}(n) \sim n^{2r+s} \preccurlyeq c_H(n)\).
		
		\item Now let \(\varepsilon=1\), and so equation \eqref{eq:twistedconj} becomes
		\begin{equation*}
			\phi(x)hx^{-1} = \lift{\Big( (l_1,\ldots,j_r)(M-I) + (x_1,\ldots,v_r) \Big)} c^{w +f(l_1,\ldots,j_r)}.
		\end{equation*}
		We think of (the Cayley graph of) $H$ embedded in $\R^{2r+s+1}$ via Mal'cev coordinates (see Definition \ref{def:malcev}), and see that the projection of a \(\phi\)-twisted conjugacy class \([h]^{\phi}\) onto the first \(2r+s\) coordinates (i.e. the image under the abelianisation map) is a coset of the image of \(M-I\). That is,
		\[\overline{[h]^{\phi}} = \bar{h} + \mathrm{Im}(M-I)\subset \Z^{2r+s}.\]
		Write \(\mathrm{rk}(M-I)\) for the dimension of the image of \(M-I\). Lemma \ref{lem:Ucosets} implies that the number of distinct images of \(\phi\)-twisted conjugacy classes in the abelianisation \(\Z^{2r+s}\) that have length at most \(n\) is \(\cO(n^{2r+s-\mathrm{rk}(M-I)})\). Note that if \(h\in H\) has length at most \(n\) then its image in the abelianisation will also have length at most \(n\) in \(\Z^{2r+s}\). So the number of distinct \(\phi\)-twisted conjugacy classes in the ball of radius \(n\) in \(H\) is \(\cO(n^{2r+s-\mathrm{rk}(M-I)+2})\), with the extra \(n^2\) coming from the growth of the \(c\) axis. We consider sub-cases depending on the dimension of the image of \(M-I\).
		\begin{enumerate}
			\item For \(2\leq\mathrm{rk}(M-I)\leq 2r+s\), we have \[c_H^{\phi}(n)=\cO(n^{2r+s-\mathrm{rk}(M-I)+2}) \preccurlyeq n^{2r+s}\preccurlyeq c_H(n)\] as required.
			
			\item Consider the case \(\mathrm{rk}(M-I)=1\). Since the kernel of \(M-I\) must be non-trivial in this case, we can fix a non-trivial element \(\mathbf{b}\in\ker(M-I)\). Then we have
			\[\left\{\bar{h}c^{w+f(\eta\mathbf{b})}\mid \eta\in\Z\right\}\subset[\bar{h}c^w]^\phi=[h]^{\phi}.\]
			Now \(f(\eta\mathbf{b})\) is a quadratic polynomial in the single variable \(\eta\) with zero constant term. Considering just those \(\eta\) in the range \([-n,n]\) yields at least \(n\) distinct elements of \(\bar{h}\langle c\rangle \cap [h]^{\phi}\), all with \(c\)-coordinate of length \(\cO(n^2)\) (since \(w=\cO(n^2)\), \(f\) has degree 2, and \(\mathbf{b}\) is constant), and hence contained in (some constant homothety of) the \(n\)-ball in \(H\). So amongst the \(\cO(n^2)\) elements of any \(\langle c\rangle\)-coset in the \(n\)-ball, there can be at most \(\cO(n)\) distinct \(\phi\)-twisted conjugacy classes.
			Combining this with the above argument again yields the claimed upper bound \[c_H^{\phi}(n)=\cO(n^{2r+s-1+1}) \preccurlyeq n^{2r+s} \preccurlyeq c_H(n).\]
			
			\item The final case is when \(\mathrm{rk}(M-I)=0\), so \(M=I\) and \(\phi\in\ker\theta\). In this specific case we can be more explicit about the function \(f\) in the calculation at \eqref{eq:twistedconj}. We have
			\begin{align*}
				\phi(x)hx^{-1} &= z_1^{l_1}\cdots b_r^{j_r}c^{k+p(l_1,\ldots,j_r)}hc^{-k}b_r^{-j_r}\cdots z_1^{-l_1} \\
				&= hc^{p(l_1,\ldots,j_r)+x_1l_1+\cdots+v_rj_r} \\
				&= \lift{(x_1,\ldots, v_r)} c^{w + (x_1+\lambda_1)l_1 + \cdots + (v_r+\lambda_{s+2r})j_r}
			\end{align*}
			where, as in the proof of Proposition \ref{prop:autoHexplicit}, \(p\) is a linear function with coefficients \(\lambda_i\) determined by \(\phi\).
			
			At each point \((x_1,\ldots,v_r)\) in the abelianisation, there are \(\gcd(x_1+\lambda_1,\ldots,v_r+\lambda_{s+2r})\) \(\phi\)-twisted conjugacy classes, all of length (asymptotically) equal to the length of \(z_1^{x_1}\cdots b_r^{v_r}\). Thus by Corollary \ref{cor:gcdoffset}, we have \(c_H^{\phi}(n)\sim c_H(n)\). 
		\end{enumerate}

	\end{enumerate}
\end{proof}

\subsection{A Conjecture}

As mentioned above, Hull and Osin \cite{OsinHull} exhibit a finitely generated group with exponential conjugacy growth function, possessing a finite index subgroup with only two conjugacy classes, and so conjugacy growth fails to be a commensurability invariant in general. On the other hand, the conjugacy growth of a virtually abelian group is equivalent to its standard growth (see Proposition \ref{prop:vabQI}), and so conjugacy growth is a quasi-isometry invariant in this very restricted context. Theorem \ref{thm:virtH} tells us that conjugacy growth is a commensurability invariant amongst class 2 nilpotent groups with infinite cyclic derived subgroup.

We ask the following natural question.
\begin{question}\label{que:qi}
	In which other classes of finitely generated groups is conjugacy growth a quasi-isometry invariant?
\end{question}
In light of Proposition \ref{prop:twistedconj}, a detailed understanding of (finite-order) automorphisms would be valuable in understanding how conjugacy growth behaves under finite extensions. The following Theorem of Pansu gives insight into the nature of quasi-isometry in nilpotent groups.
\begin{theorem}[\cite{FarbMosher},\cite{Pansu}]\label{thm:Pansu}
	For a finitely generated nilpotent group $G$, the numbers \(r_i\), the torsion-free ranks of the quotients \(G^{(i+1)}/G^{(i)}\), are quasi-isometry invariants.
\end{theorem}
This Theorem suggests that ``geometrically speaking", the numbers \(r_i\) define the structure of a finitely generated nilpotent group. In light of this, and Theorem \ref{thm:virtH}, we venture the following conjecture, which would imply that conjugacy growth was a quasi-isometry invariant amongst (virtually) nilpotent groups.
\begin{conjecture}\label{conj:nilp}
	The conjugacy growth of a finitely generated nilpotent group $G$ depends only on the numbers \(r_i\).
\end{conjecture}
A natural next step to approach this conjecture would be to extend Theorems \ref{thm:Hconjugacy} and \ref{thm:virtH} to include groups whose derived subgroup is \emph{virtually} cyclic, which would imply that conjugacy growth is a quasi-isometry invariant amongst this class of nilpotent groups.

\section*{Acknowledgments}

The author was partially supported by EPSRC grant EP/R035814/1 and would like to thank Alex Bishop, Turbo Ho, Andrei Jaikin, and Matthew Tointon for very early discussions on various aspects of this work. Thanks are also due to the anonymous referee for many useful comments and in particular for suggesting the upgrading of Proposition \ref{prop:twistedconj} from a sufficient condition to an equivalence.


\section*{Competing Interests}
The author declares none.


\bibliography{references}{}

\begin{thebibliography}{10}

\bibitem{FormalConjugacyGrowth}
Y.~Antol\'in and L.~Ciobanu.
\newblock Formal conjugacy growth in acylindrically hyperbolic groups.
\newblock {\em Int. Math. Res. Not. IMRN}, (1):121--157, 2017.

\bibitem{Babenko}
I.~K. Babenko.
\newblock Closed geodesics, asymptotic volume and the characteristics of growth
  of groups.
\newblock {\em Izv. Akad. Nauk SSSR Ser. Mat.}, 52(4):675--711, 895, 1988.

\bibitem{Bass}
H.~Bass.
\newblock The degree of polynomial growth of finitely generated nilpotent
  groups.
\newblock {\em Proc. London Math. Soc. (3)}, 25:603--614, 1972.

\bibitem{Benson}
M.~Benson.
\newblock Growth series of finite extensions of {${\bf Z}^{n}$}\ are rational.
\newblock {\em Invent. Math.}, 73(2):251--269, 1983.

\bibitem{BogMartVen}
O.~Bogopolski, A.~Martino, and E.~Ventura.
\newblock Orbit decidability and the conjugacy problem for some extensions of
  groups.
\newblock {\em Trans. Amer. Math. Soc.}, 362(4):2003--2036, 2010.

\bibitem{LinearConjugacy}
E.~Breuillard, Y.~Cornulier, A.~Lubotzky, and C.~Meiri.
\newblock On conjugacy growth of linear groups.
\newblock {\em Math. Proc. Cambridge Philos. Soc.}, 154(2):261--277, 2013.

\bibitem{SolvableConjugacy}
E.~Breuillard and Y.~de~Cornulier.
\newblock On conjugacy growth for solvable groups.
\newblock {\em Illinois J. Math.}, 54(1):389--395, 2010.

\bibitem{CEH}
L.~Ciobanu, A.~Evetts, and M.~Ho.
\newblock The conjugacy growth of the soluble {B}aumslag-{S}olitar groups.
\newblock {\em New York J. Math.}, 26:473--495, 2020.

\bibitem{CHHR}
L~Ciobanu, S.~Hermiller, D.~Holt, and S.~Rees.
\newblock Conjugacy languages in groups.
\newblock {\em Israel J. Math.}, 211(1):311--347, 2016.

\bibitem{CHM}
L.~Ciobanu, S.~Hermiller, and V.~Mercier.
\newblock {Formal conjugacy growth in graph products I}.
\newblock {\em ArXiv e-prints}, 2021.
\newblock arXiv:2103.04696.

\bibitem{ClemMajZym}
A.~E. Clement, S.~Majewicz, and M.~Zyman.
\newblock {\em The theory of nilpotent groups}.
\newblock Birkh\"{a}user/Springer, Cham, 2017.

\bibitem{DuchinSurvey}
M.~Duchin.
\newblock Counting in groups: Fine asymptotic geometry.
\newblock {\em Notices of the AMS}, 63(8):871--874, 2016.

\bibitem{DuchinShapiro}
M.~Duchin and M.~Shapiro.
\newblock The {H}eisenberg group is pan-rational.
\newblock {\em Adv. Math.}, 346:219--263, 2019.

\bibitem{Evetts}
A.~Evetts.
\newblock Rational growth in virtually abelian groups.
\newblock {\em Illinois J. Math.}, 63(4):513--549, 2019.

\bibitem{Thesis}
A.~Evetts.
\newblock {Aspects of growth in finitely generated groups}, 2020.
\newblock PhD thesis, Heriot-Watt University.

\bibitem{FarbMosher}
B.~Farb and L.~Mosher.
\newblock Problems on the geometry of finitely generated solvable groups.
\newblock In {\em Crystallographic groups and their generalizations
  ({K}ortrijk, 1999)}, volume 262 of {\em Contemp. Math.}, pages 121--134.
  Amer. Math. Soc., Providence, RI, 2000.

\bibitem{FernFern}
J.~L. Fern\'{a}ndez and P.~Fern\'{a}ndez.
\newblock Asymptotic normality and greatest common divisors.
\newblock {\em Int. J. Number Theory}, 11(1):89--126, 2015.

\bibitem{FlajoletHolonomic}
P.~Flajolet, S.~Gerhold, and B.~Salvy.
\newblock On the non-holonomic character of logarithms, powers, and the {$n$}th
  prime function.
\newblock {\em Electron. J. Combin.}, 11(2):Article 2, 16, 2004/06.

\bibitem{ContractingElements}
I.~{Gekhtman} and W.~{Yang}.
\newblock {Counting conjugacy classes in groups with contracting elements}.
\newblock {\em ArXiv e-prints}, 2018.
\newblock arXiv:1810.02969.

\bibitem{Concrete}
R.~L. Graham, D.~E. Knuth, and O.~Patashnik.
\newblock {\em Concrete mathematics}.
\newblock Addison-Wesley Publishing Company, Reading, MA, second edition, 1994.
\newblock A foundation for computer science.

\bibitem{Gromov}
M.~Gromov.
\newblock Groups of polynomial growth and expanding maps.
\newblock {\em Inst. Hautes \'Etudes Sci. Publ. Math.}, (53):53--73, 1981.

\bibitem{GubaSapir}
V.~Guba and M.~Sapir.
\newblock On the conjugacy growth functions of groups.
\newblock {\em Illinois J. Math.}, 54(1):301--313, 2010.

\bibitem{Hall}
M.~Hall, Jr.
\newblock {\em The theory of groups}.
\newblock The Macmillan Co., New York, N.Y., 1959.

\bibitem{OsinHull}
M.~Hull and D.~Osin.
\newblock Conjugacy growth of finitely generated groups.
\newblock {\em Adv. Math.}, 235:361--389, 2013.

\bibitem{Mann}
A.~Mann.
\newblock {\em How groups grow}, volume 395 of {\em London Mathematical Society
  Lecture Note Series}.
\newblock Cambridge University Press, Cambridge, 2012.

\bibitem{Margulis}
G.~A. Margulis.
\newblock Certain applications of ergodic theory to the investigation of
  manifolds of negative curvature.
\newblock {\em Funkcional. Anal. i Prilozen.}, 3(4):89--90, 1969.

\bibitem{Mercier}
V.~{Mercier}.
\newblock {Conjugacy growth series of some wreath products}.
\newblock {\em ArXiv e-prints}, 2016.
\newblock arXiv:1610.07868.

\bibitem{Osin}
D.~Osin.
\newblock Small cancellations over relatively hyperbolic groups and embedding
  theorems.
\newblock {\em Ann. of Math. (2)}, 172(1):1--39, 2010.

\bibitem{Osipov}
D.~V. Osipov.
\newblock The discrete {H}eisenberg group and its automorphism group.
\newblock {\em Mat. Zametki}, 98(1):152--155, 2015.

\bibitem{Pansu}
P.~Pansu.
\newblock M\'{e}triques de {C}arnot-{C}arath\'{e}odory et quasiisom\'{e}tries
  des espaces sym\'{e}triques de rang un.
\newblock {\em Ann. of Math. (2)}, 129(1):1--60, 1989.

\bibitem{Rivin1}
I.~Rivin.
\newblock Some properties of the conjugacy class growth function.
\newblock In {\em Group theory, statistics, and cryptography}, volume 360 of
  {\em Contemp. Math.}, pages 113--117. Amer. Math. Soc., Providence, RI, 2004.

\bibitem{Stanley}
R.~P. Stanley.
\newblock {\em Enumerative combinatorics. {V}ol. 1}, volume~49 of {\em
  Cambridge Studies in Advanced Mathematics}.
\newblock Cambridge University Press, Cambridge, 1997.
\newblock With a foreword by Gian-Carlo Rota, Corrected reprint of the 1986
  original.

\bibitem{Stoll}
M.~Stoll.
\newblock Rational and transcendental growth series for the higher {H}eisenberg
  groups.
\newblock {\em Invent. Math.}, 126(1):85--109, 1996.

\end{thebibliography}
\bibliographystyle{plain}

\bigskip

\textsc{Alex Evetts, Heilbronn Institute for Mathematical Research and Department of Mathematics, The University of Manchester, Manchester, M13 9PL, UK
}

\emph{E-mail address}{:\;\;}\texttt{alex.evetts@manchester.ac.uk}

\end{document}